\documentclass[reqno, a4paper, 11pt]{amsart}
\usepackage{amssymb}
\usepackage{mathrsfs}
\usepackage{amssymb, url, color, pb-diagram, graphicx, amscd, mathrsfs}
\usepackage[colorlinks=true, bookmarks=true, pdfstartview=FitH, pagebackref=true, linktocpage=true, linkcolor = magenta, citecolor = blue]{hyperref}
\usepackage[nodayofweek]{datetime}
\usepackage{graphicx}
\usepackage{times, txfonts}
\usepackage[pagewise,running,mathlines,displaymath]{lineno}
\usepackage[subnum]{cases}
\usepackage[compress]{cite}
\DeclareMathAlphabet{\mathpzc}{OT1}{pzc}{m}{it}
\newtheorem{theorem}{Theorem}[section]
\newtheorem{lemma}[theorem]{Lemma}
\newtheorem{corollary}[theorem]{Corollary}
\newtheorem{proposition}[theorem]{Proposition}

\theoremstyle{definition}

\theoremstyle{remark}
\newtheorem{remark}[theorem]{Remark}

\numberwithin{equation}{section}

\def\Xint#1{\mathchoice
 {\XXint\displaystyle\textstyle{#1}}%
 {\XXint\textstyle\scriptstyle{#1}}%
 {\XXint\scriptstyle\scriptscriptstyle{#1}}%
 {\XXint\scriptscriptstyle\scriptscriptstyle{#1}}%
 \!\int}
\def\XXint#1#2#3{{\setbox0=\hbox{$#1{#2#3}{\int}$}
 \vcenter{\hbox{$#2#3$}}\kern-.5\wd0}}

\def\dashint{\Xint-}

\newcommand{\definedas}{\mathrel{\raise.095ex\hbox{\rm :}\mkern-5.2mu=}}

\let\epsilon\varepsilon

\allowdisplaybreaks

\begin{document}

\title[Prescribing mean curvature on the unit ball]{Prescribing sign-changing mean curvature candidates on the $n+1$-dimensional unit ball}


\author[H. Zhang]{Hong Zhang}
\address[H. Zhang]{School of Mathematics, University of Science and Technology of China, No.96 Jinzhai Road, Hefei, Anhui, China, 230026.}
\email{\href{mailto: H. Zhang <matzhang@ustc.edu.cn>}{matzhang@ustc.edu.cn}}
\thanks{}



\date{\today \ at \currenttime}


\setpagewiselinenumbers
\setlength\linenumbersep{110pt}

\maketitle
\begin{abstract}
This paper focuses on the problem of prescribing mean curvature on the unit ball. Assume that $f$, which is allowed to change sign, satisfies Morse index counting condition or certain kind of symmetry condition. By using a negative gradient flow method, we then prove that $f$ can be realized as the boundary mean curvature of some conformal metric.
\end{abstract}
\section{introduction}
In this paper, we consider the prescribing boundary mean curvature problem on the $n+1$-dimensional unit ball with $n\geqslant2$. Such a problem is a natural analogue of the problem of prescribing scalar curvature on the sphere and it can be stated as follows. Let $(B^{n+1},g_e)$ be the $n+1$-dimensional unit ball with the Euclidean metric $g_e$. Assume that $f$ is a smooth function on the boundary $\partial B^{n+1}=S^n$. Then, one may ask if there exists a scalar-flat metric $g$ point-wisely conformally related to $g_e$, i.e., $g=u^{4/(n-1)}g_e$ for some positive and smooth function $u$, such that $f$ can be realized as the mean curvature of $g$. It is well known that this geometric problem is equivalent to finding a positive solution of the boundary value problem
\begin{equation}
  \label{mce}
  \left\{
  \begin{array}{ll}
  \Delta_{g_e}u=0,&\mbox{in}~~B^{n+1},\\
  \frac{2}{n-1}\frac{\partial u}{\partial\eta_e}+u=f(x)u^\frac{n+1}{n-1},&\mbox{on}~~S^n,
  \end{array}
  \right.
\end{equation}
where $\Delta_{g_e}$ and $\partial/\partial\eta_e$ are, respectively, the Laplace operator and the out normal derivative of the metric $g_e$.

Many research works have dealt with the Eq.\eqref{mce} during the past few decades, see, for instance \cite{ac, aco, cxy, es, eg, ho, xz}, and the references therein. Among them, Xu and the author \cite{xz}, recently, obtained the following result
\begin{theorem}[Xu \& Zhang]\label{xuzhang}
Let $n\geqslant 2$ and $f>0:S^n\rightarrow\mathbb{R}$ be a smooth Morse function satisfying the non-degeneracy condition: $|\nabla f|_{g_{S^n}}^2+|\Delta_{g_{S^n}}f|^2\neq0$ and the simple bubble condition: $\max_{S^n}f/\min_{S^n}f<\delta_n$, where $\delta_n=2^{1/n}$, $n=2$ and $\delta_n=2^{1/(n-1)}$, $n\geqslant3$. Moreover, define the numbers associated with $f$ as follows
$$m_i=\#\{\theta\in S^n; \nabla_{S^n}f(\theta)=0, \Delta_{g_{S^n}}f(\theta)<0,~~\mbox{ind}_f(\theta)=n-i\},$$
where $\mbox{ind}_f(\theta)$ denotes the Morse index of $f$ at critical point $\theta$. If the following algebraic system has no non-trivial solutions,
$$m_0=1+k_0, m_i=k_{i-1}+k_i, 1\leqslant i\leqslant n, k_n=0,$$
where coefficients $k_i\geqslant0$, then the Eq.\eqref{mce} admits at least one positive solution.
\end{theorem}
For the prescribed function $f$ possesses some kind of symmetry, Ho \cite{ho} proved a similar result as Leung \& Zhou did in \cite{lz}. Before stating his result, let us describe two types of symmetries
\begin{itemize}
 \item [(Sym1)] A mirror reflection upon a hyperplane $\mathscr{H}\subset{\mathbb{R}^{n+1}}$ passing through the origin. As the situation is invariant under a rotation, we assume that $\mathscr{H}$ is perpendicular to the $x^1$-axis. In this way, a mirror reflection $\sigma: S^n\mapsto S^n$ is given by $\sigma(x^1, x^2, \dots, x^{n+1})=(-x^1, x^2, \dots, x^{n+1})$, for $(x^1, x^2, \dots, x^{n+1})\in S^n$. As a result, $\Sigma=\{(0, x^2, \dots, x^{n+1})\in S^n\}=\mathscr{H}\cap S^n$ is the fixed point set. 
 \item [(Sym2)] A rotation of angle $\theta/k$ with axis being a straight line in $\mathbb{R}^{n+1}$ passing trough the origin and $k>1$ being an integer. We may assume that the straight line is the $x^{n+1}$-axis. In this case $\Sigma=\{\mathrm{N},\mathrm{S}\}$ is the fixed point set, where $\mathrm{N}$ is the north pole and $\mathrm{S}$ is the south pole.
\end{itemize}

Now, Ho's result can be stated as
\begin{theorem}[Ho]\label{ho}
Suppose $f>0$ is a smooth function on $S^n$ which is invariant under the symmetry {\upshape (Sym1)} or {\upshape (Sym2)}.  Assume that there exists a point $y\in\Sigma$ with $f(y)=\max_\Sigma f$ and $\Delta_{S^n}f(y)>0$. If, in addition, there holds
  \begin{equation}\label{sbc1}
  \max_{S^n}f/\max_\Sigma f<2^{1/(n-1)},
  \end{equation}
  then there exists a smooth positive solution of Eq.\eqref{mce}.
\end{theorem}

In this paper, inspired by the results in \cite{zh}, we aim to extend the Theorems \ref{xuzhang} and \ref{ho} to the case that the prescribed function $f$ is allowed to change sign. Our first result reads
\begin{theorem}
  \label{main}
  Suppose $f(x)$ is a smooth Morse function on $S^n$ satisfying the following conditions:

  \noindent {\upshape(i)}~~$\int_{S^n}f~d\mu_{S^n}>0;$

  \noindent{\upshape(ii)}~~$(\max_{S^n}|f|)\big/\big(\dashint_{S^n}f~d\mu_{S^n}\big)<2^{1/n}$;

  \noindent{\upshape(iii)}~~$|\nabla f|_{g_{S^n}}^2+(\Delta_{S^n}f)^2\neq0$;

   \noindent{\upshape(iv)}~~The following algebraic system has no non-trivial solutions
   \begin{equation}\label{Morseindex}
   m_0=1+k_0, m_i=k_{i-1}+k_i, 1\leqslant i\leqslant n, k_n=0,
   \end{equation}
   ~~~~~~~~~~~~~~\quad with coefficients $k_i\geqslant0$ and $m_i$ defined as
   \begin{equation}\label{Morseindex1}
   m_i=\#\{x\in S^n; f(x)>0, \nabla_{g_{S^n}}f(x)=0, \Delta_{g_{S^n}}f(x)<0,~~\mbox{ind}_f(x)=n-i\},
   \end{equation}
   where $\mbox{ind}_f(x)$ denotes the Morse index of $f$ at critical point $x$, then $f$ can be realized as the boundary mean curvature of some metric $g$ in the conformal class of $g_e$ on the unit ball, i.e., Eq.\eqref{mce} possesses a positive solution.
\end{theorem}
One will see that, later, the index counting condition \eqref{indexcounting} below is, indeed, a special case of the Morse index condition \eqref{Morseindex}. Hence, we have the following corollary

\begin{corollary}
  \label{cor}
  Suppose $f(x)$ is a smooth Morse function on $S^n$ satisfying the following conditions:

   \noindent {\upshape(i)}~~$\int_{S^n}f~d\mu_{S^n}>0;$

  \noindent{\upshape(ii)}~~$(\max_{S^n}|f|)\big/\big(\dashint_{S^n}f~d\mu_{S^n}\big)<2^{1/n}$;

  \noindent{\upshape(iii)}~~$|\nabla f|_{g_{S^n}}^2+(\Delta_{S^n}f)^2\neq0$;

  \noindent{\upshape(iv)}
   \begin{equation}\label{indexcounting}
   \sum_{\{x\in S^n: f(x)>0, \nabla_{S^n}f(x)=0~~\mbox{and}~~\Delta_{S^n}f(x)<0\}}(-1)^{ind_f(x)}\neq(-1)^n,
   \end{equation}
then $f$ can be realized as the boundary mean curvature of some metric $g$ in the conformal class of $g_e$ on the unit ball, i.e., Eq.\eqref{mce} possesses a positive solution.
\end{corollary}

For the case that the prescribed function $f$ possesses some kind of symmetry, we first consider a little more general situation. To do so, let us set up some notations first. We let $G$ be a subgroup of isometry group of $S^n$. Then, a function $f$ is $G-$invariant if
$$f(\theta(x))=f(x),\quad\mbox{for}~~\forall\theta\in G~~\mbox{and}~~\forall x\in S^n.$$
In addition, we define $\Sigma$ to be the fixed point set under the group $G$ as follow
$$\Sigma=\{x\in S^n: \theta(x)=x,\quad\mbox{for all}~~\theta\in G.\}.$$
our third result reads 
\begin{theorem}
  \label{main1}
 Let $G$ be a subgroup of isometry group of $S^n$. Assume that $f$ is a $G-$invariant function satisfying
  \begin{itemize}
   \item[(i)]
   $
    \int_{S^n}f~d\mu_{S^n}>0;
   $
   \item[(ii)]
    $(\max_{S^n}|f|)\big/\big(\dashint_{S^n}f~d\mu_{S^n}\big)<2^{1/n}$;
 \end{itemize}
 If there holds either
 \begin{itemize}
 \item[(a)] $\Sigma=\emptyset$ or\\ \vspace{-0.7em}
 \item[(b)] $\Sigma\neq\emptyset$ and $\max_{\Sigma}f\leqslant\dashint_{S^n}f~d\mu_{S^n}$,
 \end{itemize}
 then $f$ can be realized as the boundary mean curvature of some metric $g$ in the conformal class of $g_e$ on the unit ball, i.e., Eq.\eqref{mce} possesses a positive $G$-invariant solution.
\end{theorem}

Finally, with the help of Theorem \ref{main1}, we can extend Theorem \ref{ho} to be as follows
\begin{theorem}
 \label{main2}
 Assume that $f(x)$ is a smooth function on $S^n$ satisfying
 \begin{itemize}
  \item[(i)]$\int_{S^n}f~d\mu_{S^n}>0$, and
  \item[(ii)] $(\max_{S^n}|f|)\big/\big(\dashint_{S^n}f~d\mu_{S^n}\big)<2^{1/n}$.
 \end{itemize}
  Moreover, suppose that $f$ is invariant under the symmetry (Sym1) or (Sym2). If there exists a point $y\in\Sigma$ with $f(y)=\max_\Sigma f$ such that $\Delta_{S^n}f(y)>0$. Then Eq.\eqref{mce} possesses a positive smooth solution.
\end{theorem}

The paper is organized as follows: In \S2, we describe the evolution equations and derive some elementary estimates; In \S3, we focus ourself on the global existence of our evolutio equations; In \S4, we try to perform the blow-up analysis and describe th asymptotic behavior of the flow in the case of divergence; In the final section \S5, we will prove the main results.

\section{The flow equation and some elementary estimates}
We consider, as in Xu \& Zhang \cite{xz}, the conformal mean curvature flow as follows. Let $g(t)$ be a family of time-dependent metrics on the unit ball $B^{n+1}$ conformal to $g_e$. For simplicity, we denote by, respectively, $R(t)$ and $H(t)$ the scalar curvature and boundary mean curvature of the metric $g(t)$. Then, the evolution equation reads as
\begin{equation}\label{mee}
  \left\{
  \begin{array}{l}
  \frac{\partial}{\partial t}g(t)=-(H(t)-\lambda(t) f)g(t)\quad\mbox{on}~~\partial B^{n+1},\\
  R(t)\equiv0\quad\mbox{in}~~B^{n+1},\\
  g(0)=g_0,
  \end{array}
  \right.
\end{equation}
where $\lambda(t)$ is to be determined later and $g_0$ is the initial metric conformal to $g_e$. We would like to point out that such evolution equation as above was firstly considered by Brendle \cite{br} where the prescribed function $f$ is constant.

Now, by substituting $g(t)=u(t)^{4/(n-1)}g_e$ and $g_0=u_0^{4/(n-1)}g_e$ into \eqref{mee}, we can obtain the evolution equation for the conformal factor $u(t)$
\begin{equation}\label{eeforu}
  \left\{
  \begin{array}{l}
  \Delta_eu=0\quad\mbox{in}~~B^{n+1},\\
   \frac{\partial}{\partial t}u(t)=-\frac{n-1}{4}(H(t)-\lambda(t) f)u(t)\quad\mbox{on}~~\partial B^{n+1},\\
  u(0)=u_0,
  \end{array}
  \right.
\end{equation}
where $H(t)$ can be written, in terms of $u$, as
\begin{equation}\label{eforH}
  H=u^{2^\#-1}\bigg(a_n\frac{\partial u}{\partial\eta_e}+u\bigg).
\end{equation}
The terms $\Delta_e$ and $\partial/\partial\eta_e$ above means, respectively, the Laplace-Beltrami operator and the outer normal derivative of the metric $g_E$.

It is well known that the prescribed mean curvature problem has the associated energy functional given by
\begin{equation}\label{eforE_f[u]}
  E_f[u]=\frac{E[u]}{(\dashint_{S^n}fu^{2^\#}~d\mu_{S^n})^{2^\#/2}},
\end{equation}
where $E[u]$ can be expressed as
$$E[u]=\frac{1}{\omega_n}\int_{B^{n+1}}a_n|\nabla u|^2_{g_e}~dV_{g_e}+\dashint_{S^n}u^2~d\mu_{S^n}.$$
Here, $\omega_n$ and $d\mu_{S^n}$ are, respectively, the volume and volume form of the unit $n$-sphere. Moreover, we make a convention here that the sign `$\dashint_{S^n}$' means $\frac{1}{\omega_n}\int_{S^n}$ from now on.

Observe that, if $\Delta_eu=0$, then we have, by the divergence theorem and \eqref{eforH}, that
\begin{eqnarray}\label{eforE[u]}
  E[u]&=&\dashint_{S^n}a_n\frac{\partial u}{\partial\eta_e}u+u^2~d\mu_{S^n}=\dashint_{S^n}Hu^{2^\#}~d\mu_{S^n}\nonumber\\
  &=&\dashint_{S^n}H~d\mu_g
\end{eqnarray}
Hence, $E[u]$ is nothing but the average of the mean curvature $H$ if the metric $g$ is scalar-flat.

For the sake of convenience, we choose the factor $\lambda(t)$ in such a way that the boundary volume of $B^{n+1}$ w.r.t. the metric $g$ is preserved during the evolution, that is,
\begin{eqnarray*}
  0&=&\frac{d}{dt}vol(S^n,g(t))=\frac{d}{dt}\int_{S^n}u^{2^\#}~d\mu_{S^n}\\
  &=&\frac{n}{2}\int_{S^n}\lambda(t)f-H~d\mu_g,
\end{eqnarray*}
Thus, the choice of $\lambda(t)$ would be
\begin{equation}\label{eforlambda}
  \lambda(t)=\frac{E[u]}{\dashint_{S^n}fu^{2^\#}~d\mu_{S^n}},
\end{equation}
where we have used \eqref{eforE[u]} in the calculation.

To end this section, we collect some useful formulae which have appeared in \cite{xz}. We will omit their detailed derivation.
\begin{lemma}\label{usefulformulae}
(i) The mean curvature satisfies the evolution equation
\begin{equation}
  \label{eeforH}
  (\lambda f-H)_t=-\frac12\frac{\partial}{\partial\eta}(\lambda f-H)+\frac12(\lambda f-H)H+\lambda^\prime f
\end{equation}
on $S^n$. Here, the function $\lambda f-H$ is extended to the interior of $B^{n+1}$ such that
$$\Delta_{g(t)}(\lambda f-H)=0,\quad\mbox{in}~~B^{n+1}.$$

\noindent (ii) Let $u$ be any positive smooth solution of the flow\eqref{eeforu}. Then one has
$$\frac{d}{dt}E_f[u]=-\frac{n-1}{2}\bigg(\dashint_{S^n}f~d\mu_{g}\bigg)^{-\frac{2}{2^\#}}\dashint_{S^n}(\lambda f-H)^2~d\mu_g.$$
In particular, the energy functional $E_f[u]$ is decay along the flow.
\end{lemma}
\section{Global existence of the flow}
In this section, we will show that our flow is globally well-defined. To achieve this purpose, the critical step is to show that positive property of the quantity $\int_{S^n}fu^{2^\#}~d\mu_{S^n}$ will be preserved along the flow if we, initially, have $\int_{S^n}fu_0^{2^\#}~d\mu_{S^n}>0$.
\begin{lemma}
  \label{pos-property}
  Assume that $\int_{S^n}fu_0^{2^\#}~d\mu_{S^n}>0$ and $u$ is a smooth solution of the flow \eqref{eeforu} on $[0,T)$ for some $T>0$. Then one has
  $$\int_{S^n}fu^{2^\#}~d\mu_{S^n}>0,$$
  for all $t\in[0,T)$.
\end{lemma}
\begin{proof}
  Since $\int_{S^n}fu_0^{2^\#}~d\mu_{S^n}>0$, it is easy to see, by the definition of $E_f[u]$, that $E_f[u_0]>0$. From the sharp trace Sobolev inequality and the boundary volume-preserving property of our flow that
  \begin{equation}\label{lowerboundofE[u]}
  E[u]\geqslant\bigg(\dashint_{S^n}u_0^{2^\#}~d\mu_{S^n}\bigg)^{2/2^\#}.
  \end{equation}
  On the other hand, it follows from Lemma \eqref{usefulformulae} (ii) that
  $$\frac{E[u]}{(\dashint_{S^n}fu(t)^{2^\#}~d\mu_{S^n})^{2/2^\#}}\leqslant E_f[u_0].$$
  Combining the two inequalities above gives
  \begin{equation}
    \label{poslowerbound}
    \dashint_{S^n}fu(t)^{2^\#}~d\mu_{S^n}\geqslant\frac{\dashint_{S^n}u_0^{2^\#}~d\mu_{S^n}}{(E_f[u_0])^{2^\#/2}}>0.
  \end{equation}
\end{proof}

With help of \eqref{poslowerbound}, we can, immediately, obtain the boundedness of $\lambda(t)$.
\begin{lemma}
  \label{bdoflambda}
  Along the flow \eqref{eeforu}, $\lambda(t)$ remains bounded. To be precise, we have
  $$\lambda_1\leqslant\lambda(t)\leqslant\lambda_2,$$
  where
  $$\lambda_1=\frac{1}{\max_{S^n}f}\bigg(\dashint_{S^n}u_0^{2^\#}~d\mu_{S^n}\bigg)^{-1/n}\quad\mbox{and}
  \quad\lambda_2=\frac{(E_f[u_0])^{n/(n-1)}}{(\dashint_{S^n}u_0^{2^\#}~d\mu_{S^n})^{1/n}}.$$
\end{lemma}
\begin{proof}
  In view of \eqref{eforlambda} and \eqref{eforE_f[u]}, we may rewrite $\lambda(t)$ as
  $$\lambda(t)=E_f[u]\bigg(\dashint_{S^n}fu^{2^\#}~d\mu_{S^n}\bigg)^{-1/n}.$$
  Now, by using \eqref{poslowerbound} and the decay property of $E_f[u]$, we get
  \begin{eqnarray*}
    \lambda(t)&\leqslant& E_f[u_0]\bigg(\frac{1}{E_f[u_0]}\bigg)^{-1/(n-1)}\bigg(\dashint_{S^n}u_0^{2^\#}~d\mu_{S^n}\bigg)^{-1/n}\\
    &\leqslant&\frac{(E_f[u_0])^{n/(n-1)}}{(\dashint_{S^n}u_0^{2^\#}~d\mu_{S^n})^{1/n}}=\lambda_2.
  \end{eqnarray*}
  On the other hand, the boundary volume-preserving property and \eqref{lowerboundofE[u]} yields
  \begin{eqnarray*}
    \lambda(t)&=&\frac{E[u]}{\dashint_{S^n}fu^{2^\#}~d\mu_{S^n}}\geqslant\frac{(\dashint_{S^n}u_0^{2^\#}~d\mu_{S^n})^{2/2^\#}}
    {(\max_{S^n}f)\dashint_{S^n}u_0^{2^\#}~d\mu_{S^n}}\\
    &=&\frac{1}{\max_{S^n}f}\bigg(\dashint_{S^n}u_0^{2^\#}~d\mu_{S^n}\bigg)^{-1/n}=\lambda_1.
  \end{eqnarray*}
\end{proof}

To continue, we set
$$F_2(t)=\dashint_{S^n}(\lambda f-H)^2~d\mu_g.$$
Then we conclude that $\lambda^\prime(t)$ is bounded by $F_2(t)$.
\begin{lemma}
Let $u$ be a smooth solution of the flow \eqref{eeforu}. Then there holds
\begin{equation}
  \label{lbdaprime}
  \lambda^\prime(t)=-\bigg(\dashint_{S^n}f~d\mu_g\bigg)^{-1}\bigg[\frac{n-1}{2}\dashint_{S^n}(\lambda f-H)^2~d\mu_g+\frac12\dashint_{S^n}\lambda f(\lambda f-H)~d\mu_g\bigg].
\end{equation}
In particular,
    \begin{equation}\label{lbdaprimebound}
    |\lambda^\prime(t)|\leqslant C\Big(F_2(t)+\sqrt{F_2(t)}\Big),
    \end{equation}
    where $C>0$ is a universal constant.
\end{lemma}
\begin{proof}
  The proof of \eqref{lbdaprime} follows from a direct computation. As for \eqref{lbdaprimebound}, it follows from the H\"older's inequality, Lemma \ref{bdoflambda} and \eqref{poslowerbound}.
\end{proof}

To bound $\lambda^\prime(t)$, in view of the lemma above, it suffices to bound the quantity $F_2(t)$.
\begin{lemma}\label{upperboundofF_2}
  One can find a universal constant $C\geqslant0$ such that
  $$F_2(t)\leqslant C,$$
  for all $t\geqslant0$.
\end{lemma}
\begin{proof}
  From \eqref{eeforH} and \eqref{eeforu}, it follows that
  \begin{eqnarray*}
    \frac{d}{dt}F_2(t)&=&\frac{d}{dt}\dashint_{S^n}(\lambda f-H)^2~d\mu_g\\
      &=&2\dashint_{S^n}(\lambda f-H)\Big[-\frac12\frac{\partial}{\partial\eta}(\lambda f-H)+\frac12(\lambda f-H)H+\lambda^\prime f\Big]~d\mu_g\\
      &&+\frac{n}{2}\dashint_{S^n}(\lambda f-H)^3d\mu_g\\
      &=&-\frac{1}{\omega_n}\int_{B^{n+1}}|\nabla(\lambda  f-H)|_g^2~dV_g+\frac{2-n}{2}\dashint_{S^n}H(\lambda f-H)^2~d\mu_g\\
      &&+2\lambda^\prime\dashint_{S^n}f(\lambda f-H)~d\mu_g+\frac n2\dashint_{S^n}\lambda f(\lambda f-H)^2~d\mu_g\\
      &=&-\frac{n-2}{2}\bigg[\frac{1}{\omega_n}\int_{B^{n+1}}a_n|\nabla(\lambda f-H)|_g^2~dV_g+\dashint_{S^n}H(\lambda f-H)^2~d\mu_g\bigg]\\
      &&-\frac{1}{2\omega_n}\int_{B^{n+1}}a_n|\nabla(\lambda f-H)|_g^2~dV_g+2\lambda^\prime\dashint_{S^n}f(\lambda f-H)~d\mu_g\\
      &&+\frac n2\dashint_{S^n}\lambda f(\lambda f-H)^2~d\mu_g.
  \end{eqnarray*}
  Using the sharp Sobolev trace inequality, \eqref{lbdaprimebound} and H\"older's inequality, we obtain
  \begin{eqnarray*}
      \frac{d}{dt}F_2(t)&\leqslant&-\frac{(n-2)}{2}\bigg(\dashint_{S^n}|\lambda f-H|^{2^\#}~d\mu_g\bigg)^\frac{2}{2^\#}-\frac{1}{2\omega_n}\int_{B^{n+1}}a_n|\nabla(\lambda f-H)|_g^2~dV_g\\
      &&+2\lambda^\prime\dashint_{S^n}f(\lambda f-H)~d\mu_g+\frac n2\dashint_{S^n}\lambda f(\lambda f-H)^2~d\mu_g\\
      &\leqslant&CF_2(t)\Big(1+\sqrt{F_2(t)}\Big).
    \end{eqnarray*}
Set $$\alpha(t)=\int_0^{F_2(t)}\frac{1}{1+\sqrt{s}}~ds.$$
    Then,
    \begin{equation}\label{dalpha/dt}
    \frac{d\alpha}{dt}\leqslant CF_2(t).
    \end{equation}
    From Lemma \ref{usefulformulae} (ii) and the fact that $E_f[u]\geqslant0$, it follows that
    $$\int_0^tF_2(s)~ds\leqslant CE_f[u_0].$$
    Hence, by integrating \eqref{dalpha/dt} from $0$ to $t$ with $t>0$, we can get
    \begin{eqnarray}\label{upperboundofalpha}
      \alpha(t)\leqslant \alpha(0)+C\int_0^tF_2(s)~ds\leqslant F_2(0)+CE_f[u_0].
    \end{eqnarray}
    It is easy to see by the definition of $\alpha(t)$ that
    \begin{equation}
      \label{lowerboundofalpha}
      \alpha(t)\geqslant\frac{F_2(t)}{1+\sqrt{F_2(t)}}\geqslant
      \begin{cases}
        \frac{F_2(t)}{2}, & F_2(t)\leqslant 1,\\
        \frac{\sqrt{F_2(t)}}{2}, & F_2(t)>1.
      \end{cases}
    \end{equation}
    Combining \eqref{upperboundofalpha} and \eqref{lowerboundofalpha} yields the conclusion.
  \end{proof}
  Up to here, Lemma \ref{upperboundofF_2} and \eqref{lbdaprimebound} immediately imply the corollary
  \begin{corollary}
    \label{boundoflambda_t}
    There exists a universal constant $\Lambda_0>0$ such that
    $$
    |\lambda^\prime(t)|\leqslant\Lambda_0,
    $$
    for all $t\geqslant0$.
  \end{corollary}
Now, with the help of the boundedness of $\lambda^\prime(t)$, we are able to show the mean curvature $H$ is uniformly bounded. Before doing so, we define
$$\gamma:=\min\Bigg\{\min_{S^n}H(0)-\lambda_2\max_{S^n} |f|,-\sqrt{\tfrac43(\lambda_2\max_{S^n} |f|)^2+\tfrac83\Lambda_0\max_{S^n} |f|}\Bigg\}$$
\begin{lemma}\label{lowerboundofH}
  The mean curvature function $H$ of $g$ satisfies
  $$H-\lambda(t)f\geqslant\gamma,$$
  for all $t\geqslant0$.
\end{lemma}
  \begin{proof}
  Set
  \begin{equation*}
    \mathscr{L}=\left\{
    \begin{array}{l}
    \Delta_{g},~~\text{in}~~B^{n+1},\\
    \\
    \partial_t+\frac12\frac{\partial}{\partial\eta}+\frac12(\lambda f-\gamma),~~\text{on}~~S^n.
    \end{array}\right.
  \end{equation*}
  A simple calculation and our choice of $\gamma$ gives
  $$\mathscr{L}(\lambda f-H+\gamma)=0,$$
  in $B^{n+1}$ and on $S^n$
  \begin{eqnarray*}
    \mathscr{L}(\lambda f-H+\gamma)&=&\partial_t(\lambda f-H)+\frac12\frac{\partial}{\partial\eta}(\lambda f-H)\\
    &&+\frac12(\lambda f-\gamma)(\lambda f-H+\gamma)\\
    &=&\frac12(\lambda f)^2-\frac12\gamma^2+\frac12(\gamma-H)H+\lambda^\prime f\\
    &\le&\frac12(\lambda_2\max_{S^n} |f|)^2+\Lambda_0\max_{S^n} |f|-\frac38\gamma^2\le0.
  \end{eqnarray*}
  Moreover, it is easy to see that $\lambda f-\gamma\ge0$ and $(\lambda f-H+\gamma)(0)\le0$ due to our choice. Hence, we can apply the maximum principle to operator $\mathscr{L}$ to obtain $\lambda f-H+\gamma\le0$, which proves the assertion.
\end{proof}
Now, with the help of Lemmas \ref{pos-property}, \ref{bdoflambda} and \ref{lowerboundofH}, the proof of the global existence of the flow \eqref{eeforu} will be exactly the same as that in \cite{xz}. We thus omit the detail. Before we state this result, let us define
$$X_*=\Bigg\{0<u\in C^\infty\bigg(\overline{B^{n+1}}\bigg);
\int_{S^n}fu^{2^\#}~d\mu_{g_{S^n}}>0\Bigg\}.$$
\begin{proposition}\label{globalexistence}
  If the initial data $u_0\in X_*$, then the flow \eqref{eeforu} has a unique smooth solution which is well defined on $[0,+\infty)$.
\end{proposition}

\section{Blow-up analysis}
From this section onward, we dealt with the convergence of the flow \eqref{eeforu}. To realize this goal, one has to bound the conformal factor $u$ uniformly. However, one, in general, can not obtain this uniform bound directly. Instead, we assume the contrary, that is, the flow \eqref{eeforu} is divergent. In this way, one can expect that the blow-up phenomenon will appear and the blow-up analysis has to come into play. As an initial step, we notice the following $L^p$ convergence which is one of the key ingredients.
\begin{proposition}\label{lp}
  For $0<p<+\infty$ there holds
  $$\int_{S^n}|\lambda(t)f-H|^p~d\mu_g\rightarrow0,\qquad\mbox{as}~~t\rightarrow+\infty.$$
\end{proposition}
\begin{proof}
The proof is the same as \cite[Lemma 3.2]{xz}. We omit it here.
\end{proof}

\subsection{Compactness-Concentration}
To derive the asymptotic behavior of the flow \eqref{eeforu} in the case of divergence, the following compactness-concentration theorem in \cite{xz} serves good for our purpose. One can find its proof in \cite[Lemma 4.1]{xz}. We remark that such a compactness-concentration theorem is an analogue of that by Schwetlick \& Struwe \cite{ss}. Before we state the theorem, let us set some notations. For $r>0, x_0\in S^n$, set
$$B_r^+(x_0)=\{x\in B^{n+1}:d_{g_e}(x,x_0)<r\}\quad\mbox{and}\quad\partial^\prime B_r^+(x_0)=\partial B^+_r(x_0)\cap S^n.$$
\begin{theorem}[Xu \& Zhang]\label{cc}
Let $g_k=u_k^{4/(n-1)}g_e$, where $0<u_k\in C^\infty(\overline{B^{n+1}},g_e), k\in\mathbb{N}$, be a sequence of conformal metrics on $B^{n+1}$ with $R_{g_k}\equiv0$ and $vol(S^n,g_k)=\omega_n$. Assume that the associated boundary mean curvature $H_{g_k}$ satisfies
\begin{equation}
  \label{integralboundofHk}
  \int_{S^n}H_{g_k}~d\mu_{g_k}\leqslant C_0,\qquad\int_{S^n}\bigg|H_{g_k}-
  \int_{S^n}H_{g_k}~d\mu_{g_k}\bigg|^p~d\mu_{g_k}\leqslant C_0
\end{equation}
for all $k$ and some $p>n$. Then, either\\
{\upshape(i)} the sequence $(u_k)_k$ is uniformly bounded in $W^{1,p}(S^n,g_{S^n})$; or\\
{\upshape(ii)} there exists a subsequence of $(u_k)_k$ and finitely many points $x_1, \dots, x_L\in S^n$ such that for any $r>0$ and any $l\in\{1, \dots, L\}$ there holds
\begin{equation}\label{concentration}
  \liminf_{k\rightarrow+\infty}\bigg(\int_{\partial^\prime B_r^+(x_l)}|H_{g_k}|^n~d\mu_{g_k}\bigg)^\frac1n\geqslant\omega_n^\frac1n.
\end{equation}
\end{theorem}

To go further, let us define the so-called normalized metric and function. It is well known that for every smoothly varying family of metrics $g(t)=u(t)^{4/(n-1)}g_e$, there exists a family of conformal transformations $\phi(t):(B^{n+1},S^n)\mapsto(B^{n+1},S^n)$ which taken $S^n$ into itself such that
\begin{equation}
  \label{normalizedcondition}
  \int_{S^n}x~d\mu_h=0,\qquad\mbox{for}~~t\geqslant0,
\end{equation}
where $h=\phi^*g$ and $x=(x^1, x^2, \dots, x^{n+1})$. Here the pull-back metric $h$ is called the normalized metric such that the scalar curvature $R_h\equiv0$ in $B^{n+1}$ and the boundary mean curvature is given by $H_h=H\circ\phi$. In fact, $h$ can be expressed as $h=\varv^{4/(n-1)}g_e$ with
\begin{equation}
  \label{normalizedfunction}
  \varv=(u\circ\phi)|\det(d\phi)|^\frac{n-1}{2(n+1)},
\end{equation}
which satisfies the equation
\begin{equation}\label{eforv}
\left\{
\begin{array}{ll}
\Delta_{g_e}\varv=0&\mbox{in}~~B^{n+1}\\
a_n\frac{\partial\varv}{\partial\eta_e}+\varv=H_h\varv^{2^\#-1}&\mbox{on}~~S^n
\end{array}
\right.
\end{equation}
Since we only focus ourselves on the boundary most of time, it will be helpful to specify the formula of the conformal transformation restricted on the boundary. As a matter of fact, it is well-known that their restriction on the boundary can be written as
$$\varphi(t):=\phi(t)|_{S^n}=\pi^{-1}\circ\delta_{q(t),r(t)}\circ\pi,$$
where $\pi:S^n\backslash\{0, 0, \dots, 0, -1\}\mapsto \mathbb{R}^n$ is the stereographic projection from the south pole to $n$-plane, $\delta_{q,r}(z)=q+rz$ for $q\in\mathbb{R}^n, r>0$ and $z\in\mathbb{R}^n$.
 In consequence, we always use $\varphi$ as the conformal transformation without specific explanation. By following the idea in \cite{ms}, for $t_0\geqslant0$ fixed and $t\geqslant0$ close to $t_0$, we let
$$\varphi_{t_0}(t)=\varphi(t_0)^{-1}\varphi(t).$$
Then
$$\varphi_{t_0}(t)\circ\psi=\psi_{q(t),r(t)},$$
where $\psi=\pi^{-1}$ and $\psi_{q,r}=\psi\circ\delta_{q,r}$. This implies that at $t=t_0$, $\delta_{q(t_0),r(t_0)}(z)=z$. Hence, $q(t_0)=0, r(t_0)=1$. It is equivalent to saying that we make a translation and a dilation such that $q(t_0)=0, r(t_0)=1$ at each fixed $t_0$.

Now, given $t_0\geqslant0$, we consider a rotation mapping some $p=p(t_0)\in S^n$ into the north pole $N=(0, 0, \dots, 1)$. Then $\varphi(t_0)$ can be expressed as $\varphi(t_0)=\psi_\epsilon\circ\pi$ for some $\epsilon=\epsilon(t_0)>0$, where $\psi_\epsilon(z)=\psi(\epsilon z)=\psi_{0,\epsilon}(z)$ by the notation above. Hence, in stereographic coordinates, $\varphi(t):=\varphi_{p(t),\epsilon(t)}$ is given by
$$\varphi(t)\circ\psi
=\varphi(t_0)\circ\varphi_{t_0}(t)\circ\psi=\psi_\epsilon\circ\delta_{q,r}.$$
So, our proceeding calculations, involving any transformation, are always at each fixed time $t_0$. In this way, the conformal transformation $\varphi$ has the expression: $\varphi(t)=\psi_\epsilon\circ\pi$.

Now, we can obtain a considerably sharpen version of the previous result if we make some additional assumptions on the associated mean curvature $H_k$. We should point out that the proof of the following theorem is inspired by Struwe \cite{str}.
\begin{theorem}\label{sharpversionofcc}
 Let $(u_k)_k$ be a sequence of smooth functions on $\overline{B^{n+1}}$ with associated scalar-flat metrics $g_k=u_k^{4/(n-1)}g_e$ and the associated mean curvatures $H_k$, $k\in\mathbb{N}$. Assume that $\mbox{vol}(S^n,g_k)=\omega_n$ and there exists a smooth function $H_\infty$ on $S^n$ satisfying
 \begin{equation}
  \label{boundofH_infty}
  \max_{S^n}|H_\infty|\leqslant\tau<2^{\frac1n},
 \end{equation}
for some positive number $\tau$, such that
\begin{equation}
 \label{lpconvergestoH_infty}
 ||H_k-H_\infty||_{L^p(S^n,g_k)}\rightarrow0\quad\mbox{as}~~k\rightarrow+\infty,
\end{equation}
for some $p>n$. In addition, suppose that there exists some constant $C_*$ such that
\begin{equation}\label{lowerboundofH_k}
 H_k\geqslant C_*.
\end{equation}
Also, let $h_k=\phi_k^*g_k=\varv_k^{4/(n-1)}g_e$ be the associated sequence of normalized metrics satisfying \eqref{normalizedcondition}. Here $\phi_k$ is the conformal transformation on the unit ball and its restriction on the boundary are given by $\varphi_k=\phi_k|_{S^n}=\varphi_{p_k,\epsilon_k}$ with $p_k=p(t_k)$ and $\epsilon_k=\epsilon(t_k)$. Then, up to a subsequence, either

 \noindent{\upshape(i)} $u_k$ is uniformly bounded in $W^{1,p}(S^n,g_{S^n})$. In addition, $u_k\rightarrow u_\infty$ in $W^{1,p}(S^n,g_{S^n})$ as $k\rightarrow+\infty$. If we let $g_\infty=u_\infty^{4/(n-1)}g_e$, then $g_\infty$ has mean curvature $H_\infty$; or
 
 \noindent{\upshape(ii)} there exists a unique point $Q\in S^n$ such that
 \begin{equation}
  \label{volumeconcentration}
  d\mu_{g_k}\rightarrow\omega_n\delta_Q,
 \end{equation}
weakly in the sense of measure $~~\mbox{as}~~k\rightarrow+\infty$. Moreover, in the latter case, we have {\upshape(a)} $H_\infty(Q)=1$, and {\upshape(b)} as $k\rightarrow+\infty$, there hold 
\begin{equation}
 \label{nfandctconverge}
 ||\varv_k-1||_{C^\alpha(S^n)}\rightarrow0, ||H_k-1||_{L^p(S^n,g_k)}\rightarrow0,~~\mbox{and}~~||\varphi_k-Q||_{L^2(S^n,g_{S^n})}\rightarrow0,
\end{equation}
Here $0<\alpha<1-n/p$.
\end{theorem}
\begin{proof}
 In order to apply Theorem \ref{cc}, we need to verify
 
 \noindent(a) $\int_{S^n}H_k~d\mu_{g_k}$ is bounded, and
 
 \noindent(b)$\int_{S^n}\Big|H_k-\int_{S^n}H_k~d\mu_{g_k}\Big|^p~d\mu_{g_k}$ is bounded.
 
 \noindent To see this, from \eqref{boundofH_infty} and \eqref{lpconvergestoH_infty} it follows that
 \begin{equation}\label{integralofH_kbounded}
 \bigg|\int_{S^n}H_k~d\mu_{g_k}\bigg|\leqslant\omega_n^\frac{p-1}{p}||H_k-H_\infty||_{L^p(S^n,g_k)}+\tau\omega_n\leqslant C.
 \end{equation}
and
\begin{eqnarray}\label{integralofH_kbounded1}
&&\bigg[\int_{S^n}\bigg|H_k-\int_{S^n}H_k~d\mu_{g_k}\bigg|^p~d\mu_{g_k}\bigg]^\frac1p \nonumber\\
&&\qquad\leqslant\bigg[\int_{S^n}\big|H_k-H_\infty\big|^p~d\mu_{g_k}\bigg]^\frac1p+\bigg[\int_{S^n}\bigg|H_\infty-\int_{S^n}H_\infty~d\mu_{g_k}\bigg|^p~d\mu_{g_k}\bigg]^\frac1p\nonumber\\
&&\qquad\quad+\bigg[\int_{S^n}\bigg|\int_{S^n}H_\infty~d\mu_{g_k}-\int_{S^n}H_k~d\mu_{g_k}\bigg|^p~d\mu_{g_k}\bigg]^\frac1p\leqslant C.
\end{eqnarray}
So, (a) and (b) hold and then all conditions in Theorem \ref{cc} are satisfied. Thus, we can apply Theorem \ref{cc} to the sequence $(u_k)_k$.

\noindent(i) If $u_k$ is uniformly bounded in $W^{1,p}(S^n,g_{S^n})$ for $p>n$, then there exists $u_\infty\in W^{1,p}(S^n,g_{S^n})$ such that, up to a subsequence, $u_k\rightarrow u_\infty$ weakly in $W^{1,p}(S^n,g_{S^n})$ and strongly in $C^\alpha(S^n)$ for $0<\alpha<1-\frac np$ as $k\rightarrow+\infty$. In addition, by Sobolev embedding theory, we conclude that $||u_k||_{C^\alpha(S^n)}\leqslant C.$ Let $P=1+\sup_{k\in\mathbb{N}}\sup_{S^n}[-C_*u_k^{2/n-1}]$. Then $P$ is bounded. From the fact that $u_k>0$ and \eqref{lowerboundofH_k}, it follows that
\begin{eqnarray*}
 0&\leqslant&(H_k-C_*)u_k^\frac{n+1}{n-1}\\
 &=&a_n\frac{\partial u_k}{\partial\eta_e}+u_k-C_*u_k^\frac{n+1}{n-1}\\
 &=&a_n\frac{\partial u_k}{\partial\eta_e}+\Big(1-C_*u_k^\frac{2}{n-1}\Big)u_k\\
 &\leqslant&a_n\frac{\partial u_k}{\partial\eta_e}+Pu_k.
\end{eqnarray*}
Since $\Delta_{g_e}u_k=0$ and $\mbox{vol}(S^n,g_k)=\omega_n$, we may apply \cite[Theorem 7.2]{xz} to obtain that $u_k\geqslant C^{-1}$. Now, it follows from the assumption $||H_k-H_\infty||_{L^p(S^n,g_k)}\rightarrow0$ that $u_\infty$ weakly solves
$$
\left\{
\begin{array}{ll}
 \Delta_{g_e}u_\infty=0,&\mbox{in}~~B^{n+1},\\
 a_n\frac{\partial u_\infty}{\partial\eta_e}+u_\infty=H_\infty u_\infty^\frac{n+1}{n-1},&\mbox{on}~~S^n.
\end{array}
\right.
$$
By standard elliptic regularity and boostrapping, we conclude that $u_\infty$ is smooth since $H_\infty$ is smooth. Moreover, we have $C^{-1}\leqslant u_\infty\leqslant C$ and from 
$$a_n(\frac{\partial u_k}{\partial\eta_e}-\frac{\partial u_\infty}{\partial\eta_e})=u_\infty-u_k+(H_k-H_\infty)u_k^\frac{n+1}{n-1}+H_\infty(u_k^\frac{n+1}{n-1}-u_\infty^\frac{n+1}{n-1}),$$
it follows that $u_k\rightarrow u_\infty$ strongly in $W^{1,p}(S^n,g_{S^n})$. Since $u_\infty>0$, we may let $g_\infty=u_\infty^{4/(n-1)}g_e$. Then the metric $g_\infty$ has mean curvature $H_\infty$.\vspace{0.4em}

\noindent(ii) If the second case of Theorem \ref{cc} occurs, we then prove that this leads to the concentration behavior as described in the Theorem \ref{sharpversionofcc}. The proof will be divided into several claims.

\noindent{\bf Claim 1}: There exists only one concentration point in the sense of \eqref{concentration}.

\noindent{\itshape Proof of Claim 1}: By \eqref{boundofH_infty} and the fact that $\mbox{vol}(S^n,g_k)=\omega_n$, we can estimate
\begin{eqnarray*}
 ||H_k||_{L^n(S^n,g_k)}&\leqslant&||H_k-H_\infty||_{L^n(S^n,g_k)}+\tau\omega_n^\frac1n.
\end{eqnarray*}
This together with \eqref{lpconvergestoH_infty} implies that
$$\liminf_{k\rightarrow+\infty}\int_{S^n}|H_k|^n~d\mu_{g_k}\leqslant\tau^n\omega_n<2\omega_n.$$
 Now, suppose $\{x_1, \dots, x_m \}$, defined in the Theorem \ref{cc}, are concentration points with $m\geqslant2$. Let $0<r<\frac12\{\mbox{dist}(x_i,x_j);1\leqslant i<j\leqslant m\}$. It follows from \eqref{concentration} and the estimate above that
  \begin{eqnarray*}
    m&\leqslant&\sum_{i=1}^m\liminf_{k\rightarrow+\infty}\omega_n^{-1}
    \int_{\partial^\prime B_r(x_i)}|H_k|^\frac n2~d\mu_k\\
    &\leqslant&\liminf_{k\rightarrow+\infty}\bigg[\sum_{i=1}^m\omega_n^{-1}\int_{\partial^\prime B_r(x_i)}|H_k|^n~d\mu_k\bigg]\\
    &\leqslant&\liminf_{k\rightarrow+\infty}\omega_n^{-1}\int_{\cup_{i=1}^m\partial^\prime B_r(x_i)}|H_k|^n~d\mu_k\\
    &\leqslant&\liminf_{k\rightarrow+\infty}\dashint_{S^n}|H_k|^n~d\mu_k<2,
  \end{eqnarray*}
  which implies that $m<2$ and thus contradicts with $m\geqslant2$. This shows that $m=1$, i.e. concentration point is unique. 
 
\noindent{\bf Claim 2}: There exists a constant $C>0$ such that $C^{-1}\leqslant\varv_k\leqslant C$.

\noindent{\itshape Proof of Claim 2}: For the normalized sequence $(\varv_k)_k$, we note that
  $$\int_{S^n}H_k\circ\varphi_k~d\mu_{h_k}=\int_{S^n}H_k~d\mu_{g_k}$$
  $$\int_{S^n}\bigg|H_k\circ\varphi_k-\int_{S^n}H_k\circ\varphi_k~d\mu_{h_k}\bigg|^p~d\mu_{h_k}
  =\int_{S^n}\bigg|H_k-\int_{S^n}H_k~d\mu_{g_k}\bigg|^p~d\mu_{g_k},$$
  $$R_{h_k}=R_k\circ\phi_k\equiv0\quad\mbox{and}\quad vol(S^n,h_k)=vol(S^n,g_k)=\omega_n.$$
  Therefore, it follows from \eqref{integralofH_kbounded} and \eqref{integralofH_kbounded1} that all the conditions in Theorem \ref{cc} hold for $(h_k)_k$. This means that we can apply Theorem \ref{cc} to the sequential metrics $(h_k)_k$. We claim that the case (ii) in Theorem \ref{cc} can never occur to $(h_k)_k$. Suppose the contrary. Then, we may follow the exact same proof as that in {\bf Claim 1} above to get that there exists a unique point $Q$ such that \eqref{concentration} holds for $(h_k)_k$. So, for sufficiently large $k$ and any $r>0$, we have, by \eqref{boundofH_infty} and \eqref{lpconvergestoH_infty}, that
\begin{eqnarray*}
    &&\omega_n^\frac 1n+o(1)\leqslant\bigg(\int_{\partial^\prime B_r(Q)}|H_{h_k}|^n~d\mu_{h_k}\bigg)^\frac 1n\\
    &&\quad\leqslant\bigg(\int_{\partial^\prime B_r(Q)}|H_{h_k}-H_\infty\circ\varphi_k|^n~d\mu_{h_k}\bigg)^\frac 1n+\max_{S^n}\big|H_\infty\big|\bigg(\int_{\partial^\prime B_r(Q)}~d\mu_{h_k}\bigg)^\frac 1n\\
    &&\quad\leqslant o(1)+\tau
    \bigg(\int_{\partial^\prime B_r(Q)}~d\mu_{h_k}\bigg)^\frac 1n,
  \end{eqnarray*}
   which implies that
  \begin{equation*}
    \int_{\partial^\prime B_r(Q)}~d\mu_{h_k}\geqslant \tau^{-n}\omega_n+o(1).
  \end{equation*}
  Since $vol(S^n,h_k)=\omega_n$, we get
  \begin{equation}\label{v1}
    \int_{S^n\backslash \partial^\prime B_r(Q)}~d\mu_{h_k}\leqslant(1-\tau^{-n})\omega_n+o(1).
  \end{equation}
  From \eqref{v1}, it follows that
  \begin{eqnarray*}
    \bigg|\Big|\dashint_{S^n}x~d\mu_{h_k}\Big|-1\bigg|&=&\bigg|\Big|\dashint_{S^n}x~d\mu_{h_k}\Big|-|Q|\bigg|\\
    &\leqslant&\bigg|\dashint_{S^n}x~d\mu_{h_k}-Q\bigg|\leqslant\dashint_{S^n}|x-Q|~d\mu_{h_k}\\
    &=&\omega_n^{-1}\int_{S^n\backslash \partial^\prime B_r(Q)}|x-Q|~\varv_k^{2^\#}d\mu_{S^n}+\omega_n^{-1}\int_{\partial^\prime B_r(Q)}|x-Q|~\varv_k^{2^\#}d\mu_{S^n}\\
    &\leqslant&2(1-\tau^{-n})+r+o(1).
  \end{eqnarray*}
  Notice that $\tau<2^{1/n}$. This implies that $2\tau^{-n}-1>0$. Now, by choosing $r=(2\tau^{-n}-1)/2$ and $k$ large enough, we then have
  \begin{eqnarray*}
    \bigg|\dashint_{S^n}x~d\mu_{h_k}\bigg|&\geqslant&1-2(1-\tau^{-n})-r+o(1)\\
    &=&\frac{2\tau^{-n}-1}{2}+o(1)>0.
  \end{eqnarray*}
  However, this contradicts with the fact that $h_k$ satisfies \eqref{normalizedcondition}. Such a contradiction shows that case (i) in Theorem \ref{cc} will happen to $(h_k)_k$, that is, $\varv_k$ is uniformly bounded in $W^{1,p}(S^n,g_{S^n})$ for $p>n$. By Sobolev embedding theory, we conclude that there exists a positive constant $C$ such that $||\varv_k||_{C^\alpha(S^n)}\leqslant C$ for $0<\alpha<1-n/p$. Let
  $$P:=1+\sup_{k\in\mathbb{N}}\sup_{S^n}[-C_*\varv_k^{4/(n-1)}].$$
Then by using the facts that $\varv_k>0$ and $C_*\leqslant H_k\circ\phi_k$, we follow the same proof as before to get
$$
\left\{
  \begin{array}{ll}
  \Delta_{g_e}\varv_k=0,&\mbox{in}~~B^{n+1},\\
a_n\frac{\partial \varv_k}{\partial\eta_e}+P\varv_k\geqslant0,&\mbox{on}~~S^n.
  \end{array}
\right.  
$$
  From \cite[Theorem 7.2]{xz} and the fact that $\int_{S^n}\varv_k^{2^\#}~d\mu_{S^n}=\omega_n,$
  it follows that $\varv_k\geqslant C^{-1}$.
 
\noindent{\bf Claim 3}: $\varv_k\rightarrow1$ in $C^\alpha(S^n)$ with $0<\alpha<1-n/p$.

\noindent{\itshape Proof of Claim 3}: Since $\varv_k$ is bounded in $W^{1,p}(S^n,g_{S^n})$ by the proof of Claim 2, it follows from the Sobolev embedding theory that there exists a function $\varv_\infty\in W^{1,p}(S^n,g_{S^n})$ such that, up to a subsequence, $\varv_k\rightarrow \varv_\infty$ weakly in $W^{1,p}(S^n,g_{S^n})$ and strongly in $C^\alpha(S^n)$ with $0<\alpha<1-n/p$.

 Since $(p_k)_k\subset S^n$, we may assume that $p_k\rightarrow Q_*$ for some $Q_*\in S^n$.  Up to here, we claim that the parameter $\epsilon_k$ in the conform transformation $\varphi_k$ satisfies $\epsilon_k\rightarrow0$ as $k\rightarrow+\infty$. If not, we may assume that $\epsilon_k\rightarrow \epsilon_0>0$, then $\varphi_k\rightarrow\varphi_{Q_*,\epsilon_0}$ and thus $\det(d\varphi_k)\rightarrow\det(d\varphi_{Q_*,\epsilon_0})$ which is bounded away from zero. From this fact, the assertion of {\bf Claim 2} and \eqref{normalizedfunction}, it follows that $u_k$ is bounded from both below and above by a positive constant. Now, in view of the equation
  $$a_n\frac{\partial u_k}{\partial\eta_e}+u_k=H_ku_k^{2^\#-1},$$
we can get that $u_k$ is uniformly bounded in $W^{1,p}(S^n,g_{S^n})$, which contradicts with our assumption. Hence, we have $\epsilon_k\rightarrow0$. By the definition of $\varphi_k$, we conclude that $\varphi_k\rightarrow Q_*$ for $x\in S^n\backslash\{-Q_*\}$. This fact together with \eqref{lpconvergestoH_infty}, {\bf Claim 2} and the dominated convergence theorem implies that
 \begin{eqnarray}
   \label{H_klpconvergence}
   ||H_{h_k}-H_\infty(Q_*)||_{L^p(S^n,h_k)}&\leqslant&||H_{h_k}-H_\infty\circ\varphi_k||_{L^p(S^n,h_k)}\nonumber\\
   &&+||H_\infty\circ\varphi_k-H_\infty(Q_*)||_{L^p(S^n,h_k)}\rightarrow0
 \end{eqnarray}
  as $k\rightarrow+\infty$. This implies that $\varv_\infty$ weakly solves
  \begin{equation*}
  \left\{
  \begin{array}{ll}
    \Delta_{g_e}\varv_\infty=0&\mbox{in}~~B^{n+1},\\
    a_n\frac{\partial\varv_\infty}{\partial\eta_e}+\varv_\infty=H_\infty(Q_*)\varv_\infty^{2^\#-1}&\mbox{on}~~S^n.
  \end{array}
  \right.
  \end{equation*}
  Since we have $\int_{S^n}x\varv_k^{2^\#}~d\mu_{S^n}$=0 and $\int_{S^n}\varv_k^{2^\#}~d\mu_{S^n}=\omega_n$, $\varv_\infty$ will satisfy $\int_{S^n}x\varv_\infty^{2^\#}~d\mu_{S^n}$=0 and $\int_{S^n}\varv_\infty^{2^\#}~d\mu_{S^n}=\omega_n$. It follows, by the Escobar's uniqueness theorem, that $\varv_\infty$ must be a constant and hence $\varv_\infty\equiv1$. Moreover, plugging $\varv_\infty\equiv1$ into the equation above yields $H_\infty(Q_*)=1$. Finally, substituting $H_\infty(Q^*)=1$ into \eqref{H_klpconvergence} gives 
  \begin{equation}\label{H_kconvergesto1}
  ||H_k-1||_{L^p(S^n,g_k)}\rightarrow0,\quad\mbox{as}~k\rightarrow+\infty.
  \end{equation}
 
 \noindent{\bf Claim 4}: There exists a unique point $Q\in S^n$ such that $d\mu_{g_k}\rightarrow\omega_n\delta_Q$ weakly in the sense of measure as $k\rightarrow+\infty$.
 
 \noindent{\itshape Proof of Claim 4}: It follows from \eqref{H_kconvergesto1}, {\bf Claim 1} and \eqref{concentration} that there exists a unique point $Q\in S^n$ such that, for $k$ large enough and any $r>0$, there holds
 \begin{eqnarray*}
   &&\omega_n^\frac 1n+o(1)\leqslant\bigg(\int_{\partial^\prime B_r(Q)}|H_k|^n~d\mu_{k}\bigg)^\frac 1n\\
    &&\quad\leqslant\bigg(\int_{\partial^\prime B_r(Q)}|H_k-1|^n~d\mu_{k}\bigg)^\frac 1n+\bigg(\int_{\partial^\prime B_r(Q)}~d\mu_{k}\bigg)^\frac 1n\\
    &&\quad=o(1)+\bigg(\int_{\partial^\prime B_r(Q)}~d\mu_{k}\bigg)^\frac 1n\leqslant \omega_n^\frac 1n+o(1).
\end{eqnarray*}
Hence, $d\mu_k\rightarrow\omega_n\delta_Q,$ weakly in the sense of measure as $k\rightarrow+\infty$.
 
\noindent{\bf Claim 5}: There hold (a) $H_\infty(Q)=1$ and (b) $||\varphi_k-Q||_{L^2(S^n,g_{S^n})}\rightarrow0$ as $k\rightarrow+\infty$.

\noindent{\itshape Proof of Claim 5}: In view of the proof of {\bf Claim 3}, it suffices to show that $Q_*=Q$. Indeed, 
 on one hand, from {\bf Claim 4} it follows that
$$\dashint_{S^n}x~d\mu_k\rightarrow Q,\quad\mbox{as}~~k\rightarrow+\infty.$$
On the other hand, it follows from the fact that $\varv_k$ is uniformly bounded and the dominated convergence theorem that
  \begin{eqnarray*}
    \bigg|\dashint_{S^n}\varphi_kd\mu_{h_k}-Q_*\bigg|&\leqslant&\dashint_{S^n}|\varphi_k-Q_*|d\mu_{h_k}\rightarrow0,
  \end{eqnarray*}
  as $k\rightarrow+\infty$. Notice that, by the change of variables, one has
  $$\dashint_{S^n}x~d\mu_k=\dashint_{S^n}\varphi_kd\mu_{h_k}.$$
Now, it is easy to see that $Q_*=Q$.
\end{proof}

\subsection{Blow-up analysis}\label{blowup}
The principal goal of this subsection is to describe the sequential asymptotic behavior of the flow 
\eqref{eeforu}. Firstly, let us recall that $f$ satisfies $\max_{S^n}|f|/(\dashint_{S^n}f~d\mu_{S^n})<2^{1/n}$. Hence, we can choose
$$\sigma=\frac12\bigg[2^\frac1n
\bigg(\frac{\dashint_{S^n}f~d\mu_{S^n}}{\max_{S^n}|f|}\bigg)-1\bigg]>0,$$
and set
\begin{equation}\label{beta}
\beta=(1+\sigma)^\frac{n-1}{n}
\bigg(\dashint_{S^n}f~d\mu_{S^n}\bigg)^\frac{1-n}{n}.
\end{equation}
With all notations above settled, we finally define the set
$$X_f=\bigg\{u\in X_*;\int_{S^n}u^{2^\#}~d\mu_{S^n}=\omega_n~~\mbox{and}~~E_f[u]\leqslant\beta\bigg\}.$$
\begin{remark}\label{rk2}
  Notice that $X_f\neq\emptyset$, In fact, when $u\equiv1$ we have $\int_{S^n}u^{2^\#}=\omega_n$, $\dashint_{S^n}fu^{2^\#}~d\mu_{S^n}=\dashint_{S^n}f~d\mu_{S^n}>0$ and $E_f[u]=(\dashint_{S^n}f~d\mu_{S^n})^{(1-n)/n}<\beta$. Hence, $u\equiv1\in X_f$.
\end{remark}

Then for any $u_0\in X_f$, it follows from Proposition \ref{globalexistence} that the flow \eqref{eeforu} has a unique smooth solution $u(t)$ well defined on $[0,+\infty)$. For an arbitrary time sequence $(t_k)_k\subset[0,+\infty)$ with $t_k\rightarrow+\infty$ as $k\rightarrow+\infty$, we set
$$u_k=u(t_k),\quad g_k=g(t_k),\quad H_k=H_{g_k}\quad\mbox{and}\quad d\mu_k=d\mu_{g_k}.$$
In order to apply Theorem \ref{sharpversionofcc} to our flow, we have to verify all conditions in that theorem are satisfied.
\begin{lemma}
For the sequence $(u_k)_k$ with the associated mean curvatures $H_k$ defined as bove, there hold

\noindent{\upshape(i)} $vol(S^n,g_k)=\omega_n$ and $R_k\equiv0$.

\noindent{\upshape(ii)} there exists a smooth function $H_\infty$ with $\max_{S^n}\big|H_\infty\big|\leqslant\tau<2^{1/n}$ such that $||H_k-H_\infty||_{L^p(S^n,g_k)}\rightarrow0$, for some $p>n$;

\noindent{\upshape(iii)} there exists a sequence of smooth functions $\sigma_k(x)$ with $\sup_{k\in\mathbb{N}}||\sigma_k||_{C^0(S^n)}\leqslant C_*$ for some positive number $C_*$ such that $\sigma_k\leqslant H_k$.
\end{lemma}
\begin{proof}
 (i) from the choice of the initial data $u_0$ and the volume-preserving property of the flow \eqref{eeforu}, it follows that $vol(S^n,g_k)=\omega_n$. Moreover, the flow equation \eqref{mee} shows that $R_k\equiv0$.

 \noindent(ii) By Lemma \ref{bdoflambda}, we may assume that, up to a subsequence, $\lambda(t_k)\rightarrow\lambda_\infty$ as $k\rightarrow+\infty$, where $\lambda_\infty\in[\lambda_1,\lambda_2]$. Now, by the definition of $\lambda_2$ and the choices of initial data $u_0$ and $\beta$, we can estimate
 \begin{eqnarray*}
  \lambda_\infty|f|&\leqslant&\lambda_2\max_{S^n}|f|\leqslant(E_f[u_0])^\frac{n}{n-1}\max_{S^n}|f|\\
  &\leqslant&\beta^\frac{n}{n-1}\max_{S^n}|f|=(1+\sigma)\frac{\max_{S^n}|f|}{\dashint_{S^n}f~d\mu_{S^n}}.
  \end{eqnarray*}
By setting $\tau=(1+\sigma)\max_{S^n}|f|/\dashint_{S^n}f~d\mu_{S^n}$, we obtain $\lambda_\infty|f|\leqslant\tau<2^{1/n}$. Moreover, from Proposition \ref{lp}, it follows that
\begin{eqnarray*}
 ||H_k-\lambda_\infty f||_{L^p(S^n,g_k)}&\leqslant&||H_k-\lambda(t_k)f||_{L^p(S^n,g_k)}\\
 &&+|\lambda(t_k)-\lambda_\infty|||f||_{L^p(S^n,g_k)}\rightarrow0,
\end{eqnarray*}
for some $p>n$. By setting $H_\infty=\lambda_\infty f$, we thus complete the proof of (i). 
 
\noindent(ii) By Lemma \ref{lowerboundofH}, we have $H_k\geqslant\lambda(t_k)f+\gamma$. From Lemma \ref{bdoflambda}, we conclude that $\lambda(t_k)f\geqslant-\lambda_2\max_{S^n}|f|$. By setting $C_*=-\lambda_2\max_{S^n}|f|+\gamma$, we thus obtain $H_k\geqslant C_*$ 
\end{proof}

From this lemma, a direct consequence of Theorem \ref{sharpversionofcc} reads
\begin{corollary}
 \label{sequentialblowup}
 Let $u(t)$ be the smooth solution of the flow \eqref{eeforu} with initial data $u_0\in X_f$. Associated with the sequential metrics $g_k=u_k^{4/(n-1)}g_e$, we let $h_k=\phi_k^*g_k=\varv_k^{4/(n-1)}g_e$ be the sequence of corresponding normalized metrics, where $\phi_k$ is the conformal transformation on the unit ball and its restriction on the boundary are given by $\varphi_k=\phi_k|_{S^n}=\varphi_{p_k,\epsilon_k}$ with $p_k=p(t_k)$ and $\epsilon_k=\epsilon(t_k)$. Then there hold either
 
 \noindent{\upshape(i)} there exists a positive function $u_\infty\in W^{1,p}(S^n,g_{S^n})$ such that $u_k\rightarrow u_\infty$ in $W^{1,p}(S^n,g_{S^n})$. In addition, if we let $g_\infty=u_\infty^{4/(n-1)}g_e$ then $g_\infty$ has mean curvature $\lambda_\infty f$; or
 
 \noindent{\upshape(ii)} there exists a uniqueness point $Q$ (depending only on $u_0$) such that $d\mu_k\rightarrow\omega_n\delta_Q$ weakly in the sense of measure as $k\rightarrow+\infty$. Moreover, in the latter case, one has {\upshape(a)} $\lambda_\infty f(Q)=1$. In particular, $f(Q)>0$, and {\upshape(b)} $||\varv_k-1||_{C^\alpha(S^n)}\rightarrow0$ with $0<\alpha<1-n/p$, $||H_k-1||_{L^p(S^n,g_k)}\rightarrow0$ and $||\varphi_k-Q||_{L^2(S^n,g_{S^n})}\rightarrow0$ as $k\rightarrow+\infty$.
\end{corollary}

\subsection{Asymptotic behavior of the flow}
In view of Corollary \ref{sequentialblowup}, to prove our theorems, it suffices to prove that there exists a time sequence $(t_k)_k$ such that case (i) occurs to the corresponding metrics $(g_k)_k$. However, this assertion is generally hard to be obtained directly. Here, we adopt the contradiction argument. So, we assume that $f$ can not be realized as a mean curvature of any conformal metric. In other words, the flow will be divergent and then case (ii) happens to the metrics $(g_k)_k$ for arbitrary time sequence $(t_k)_k$. In this scenario, we can pass the sequential behavior to the uniform one of the flow \eqref{eeforu}.
For $t\geqslant0$, let
$$S=S(t)=\dashint_{S^n}x~d\mu_g$$
be the center of mass of $g=g(t)$. Then we have
\begin{lemma}\label{S(t)not0}
  $S(t)\rightarrow Q$ as $t\rightarrow+\infty$. In particular, $S(t)\neq0$ for all large $t$.
\end{lemma}
\begin{proof}
  For an arbitrary time sequence $(t_k)_k\subset[0,+\infty)$ with $t_k\rightarrow+\infty$ as $k\rightarrow+\infty$, we can apply Corollary \ref{sequentialblowup} to the sequential metrics $(g(t_k))_k$ to get that $d\mu_{g_k}\rightarrow\omega_n\delta_Q$ as $k\rightarrow+\infty$. Hence
  $S(t_k)\rightarrow Q$ as $k\rightarrow+\infty$. By the arbitrariness of the sequence $(t_k)_k$, we thus conclude that $S(t)\rightarrow Q$ as $t\rightarrow+\infty$.
\end{proof}
In view of this lemma, we may assume that $S(t)\neq0$ for all $t\geqslant0$. Then the image of $S(t)$ under radial projection
$$Q(t)=S/|S|\in S^n$$
is well defined for all $t\geqslant0$.
\begin{proposition}\label{asymptoticbehavior}
  Suppose that $f$ can not be realized as the mean curvature of any conformal metric on the boundary $S^n$. Let $u(t)$ be the smooth solution of \eqref{eeforu}, $\varphi(t)$ the restriction of the conformal transformation $\phi(t)$ on the boundary, $\varv(t)$ the corresponding normalized flow and $h(t)$ the normalized metric. Then, as $t\rightarrow+\infty$, there hold
  \begin{itemize}
    \item[(i)] $\max_{S^n}u(\cdot,t)\rightarrow+\infty$;
    \item[(ii)] $\varv(t)\rightarrow1$ in $C^{\alpha}(S^n)$ for $\alpha\in(0,1)$, $\varphi(t)\rightarrow Q$ in $L^{2}(S^n,g_{S^n})$, and\\ $||H(t)-1||_{L^p(S^n,g(t))}\rightarrow0$ for $p>n$;\vspace{0.2em}
    \item[(iii)] $ Q(t)\rightarrow Q$ and $\lambda(t)f(Q(t))\rightarrow 1$;
    \item[(iv)] $E_f[u(t)]\rightarrow (f(Q))^\frac{1-n}{n}$, and furthermore,
    \item[(v)] $\nabla_{S^n}f(Q)=0$ and $\Delta_{S^n}f(Q)<0$.
  \end{itemize}
\end{proposition}
\begin{proof}
  (i) Assume that $||u(t)||_{C^0(S^n)}\leqslant C$ for some positive number $C$ and all $t\geqslant0$. Observe that 
  $$a_n\frac{\partial}{\partial\eta_e}u(t)+u(t)=(H(t)-\lambda(t)f)u(t)^\frac{n+1}{n-1}+\lambda(t)fu(t)^\frac{n+1}{n-1}.$$
  Hence, by Lemma \ref{bdoflambda} and Proposition \ref{lp}, it is easy to see that $u(t)$ is uniformly bounded in $W^{1,p}(S^n,g_{S^n})$. Then by Corollary \ref{sequentialblowup}, we can obtain that, up to a constant, $f$ can be realized as a mean curvature of some conformal metric. But this contradicts with our assumption.

  (ii) By way of contradiction, we assume that there exists a time sequence $t_k\rightarrow+\infty$ such that
  $$\liminf_{k\rightarrow+\infty}\big(||\varv_k-1||_{C^\alpha(S^n)}+||\varphi_k-Q||_{L^2(S^n,g_k)}+||H_k-1||_{L^p(S^n,g_k)}\big)>0.$$
  But then, by Corollary \ref{sequentialblowup}, a subsequence $u_k\rightarrow u_\infty$ with $g_\infty=u_\infty^{4/(n-1)}g_e$. The metric $g_\infty$, up to a constant, has the mean curvature $f$, contrary to our assumption.
  
  (iii) It follows from Lemma \ref{S(t)not0} and the definition of $Q(t)$ that $Q(t)\rightarrow Q$. This together with the fact that $\dashint_{S^n}H_h~d\mu_h\rightarrow1,$ and Proposition \ref{lp} implies that
  \begin{eqnarray*}
   &&\lim_{t\rightarrow+\infty}\Big(1-\lambda(t)f(Q(t))\Big)=\lim_{t\rightarrow+\infty}\bigg[\dashint_{S^n}H_h~d\mu_h-\lambda(t)f(Q(t))\bigg]\\
   &&\qquad=\lim_{t\rightarrow+\infty}\bigg[\dashint_{S^n}H_h-\lambda(t)f\circ\varphi(t)~d\mu_h+\lambda(t)\dashint_{S^n}f\circ\varphi(t)-f(Q)~d\mu_h\\
   &&\qquad\quad+\lambda(t)\Big(f(Q)-f(Q(t))\Big)\bigg]=0
  \end{eqnarray*}

  (iv) Recall that 
   $$E_f[u(t)]=\frac{\dashint_{S^n}H(t)~d\mu_{g(t)}}{\Big(\dashint_{S^n}f~d\mu_{g(t)}\Big)^\frac{n-1}{n}}.$$
  On one hand, $\dashint_{S^n}H(t)~d\mu_{g(t)}\rightarrow1$; On the other hand, 
  $$\bigg|\dashint_{S^n}f~d\mu_{g(t)}-f(Q)\bigg|\leqslant\dashint_{S^n}\big|f\circ\varphi(t)-f(Q)\big|~d\mu_h\rightarrow0.$$
  Hence, the assertion holds.
  
  (v) Since the proof is exactly the same as that in \cite[Proposition 5.7]{xz}, we omit the details.
\end{proof}
\section{Proof of Theorems \ref{main} and \ref{main1}}
\subsection{Proof of Theorem \ref{main}}
Recall the notations in subsection \ref{blowup}, for $p\in S^n$, $0<\epsilon<+\infty$, if we put the point $p$ at the origin in stereographic coordinates, then the restriction of conformal transform $\phi$ on the boundary $S^n$ is given by $\varphi_{p,\epsilon}=\psi_\epsilon\circ\pi$. Let $g_{p,\epsilon}|_{S^n}=u^{4/(n-1)}_{p,\epsilon}g_{S^n}$ with $u_{p,\epsilon}=|\det(d\varphi_{p,\epsilon})|^{1/2^\#}$. Then
$$d\mu_{g_{p,\epsilon}}\rightarrow\omega_n\delta_p,\qquad\mbox{as}~~\epsilon\rightarrow0.$$
Now, for $\rho\in\mathbb{R}_+$, we denote the sub-level set of $E_f$ by
$$L_\rho=\{u\in X_f:E_f[u]\leqslant\rho\}.$$
It follows from Proposition \ref{asymptoticbehavior} that the concentration phenomenon can only occur at the critical points of $f$ where it takes positive values. For convenience, we label all these critical points $p_1, \dots, p_N$ of $f$ so that $0<f(p_i)\leqslant f(p_j)$ for $1\leqslant i\leqslant j\leqslant N$ and let
$$\beta_i=(f(p_i))^\frac{1-n}{n}=\lim_{\epsilon\rightarrow0}E_f[u_{p_i,\epsilon}], 1\leqslant i\leqslant N.$$
We may assume, w.l.o.g., that all positive critical levels $f(p_i)$, $1\leqslant i\leqslant N$ are distinct. By choosing $s_0=\frac{1}{3}\min_{i\leqslant i\leqslant N-1}\{\beta_i-\beta_{i+1}\}>0$, we then have $\beta_i-2s_0>\beta_{i+1}$ for all $i$.
Now, we are ready to characterize the homotopy on $L_\rho$. We remark here that the contraction mapping given by Xu \& Zhang \cite{xz} does not work anymore. We need a new construction for such a mapping.
\begin{proposition}\label{homotopy}
~~\\
  {\upshape(i)} If $\max\Big\{\beta_1,(\dashint_{S^n}f~d\mu_{S^n})^{(1-n)/n}\Big\}<\beta_0\leqslant\gamma$, where $\beta$ has been chosen in \eqref{beta}, then $L_{\beta_0}$ is contractible.\\
  {\upshape(ii)} For $0<s\leqslant s_0$ and each $i$, the set $L_{\beta_i-s}$ is homotopy equivalent to the set $L_{\beta_{i+1}+s}$.\\
  {\upshape(iii)} For each critical point $p_i$ of $f$ with $\Delta_{S^n}f(p_i)>0$, the set $L_{\beta_i+s_0}$ is homotopy equivalent to the set $L_{\beta_i-s_0}$.\\
  {\upshape(iv)} For each critical point $p_i$ of $f$ with $\Delta_{S^n}f(p_i)<0$, the set $L_{\beta_i+s_0}$ is homotopy equivalent to the set $L_{\beta_i-s_0}$ with $(n-\mbox{ind}_f(p_i))$-cell attached.
\end{proposition}
\begin{proof}
  (i). For each $u_0\in L_{\beta_0}$, we fix a sufficiently large $T>0$ and set $\zeta=\zeta(T)=(\max_{S^n}u(T,u_0))$. It follows from Proposition \ref{asymptoticbehavior} (i) that
  \begin{equation}
    \label{zetagoesto0}
    \lim_{T\rightarrow+\infty}\zeta=0.
  \end{equation}
 In view of the proof of \cite[Lemma 6.2]{xz}, $T$ can be chosen continuously depending on the initial data $u_0$. So, $\zeta$ is continuously depending on $u_0$ too. Define
 $$
 u_s=
 \left\{
 \begin{array}{ll}
u(2sT,u_0),&0\leqslant s\leqslant\frac12,\\
\bigg[\frac{(2-2s)(\zeta u(T,u_0))+2s-1}{(2-2s)\zeta^{2^\#}+2s-1}\bigg]^\frac{1}{2^\#},&\frac12<s\leqslant1.
 \end{array}
 \right.
 $$
 Then, we have the claim

 \noindent{\bf Claim}: The function $u_s$ satisfies $1^\circ.\dashint_{S^n}u_s^{2^\#}~d\mu_{S^n}=1$, $2^\circ. \dashint_{S^n}fu_s^{2^\#}~d\mu_{S^n}>0$ and $3^\circ. E_f[u_s]\leqslant\beta_0$.

 \noindent{\itshape Proof of Claim}: From the choice of the initial data $u_0$, the volume-preserving property of the flow \eqref{eeforu}, Lemma \ref{pos-property} and the decay property of the energy functional $E_f[u]$, it follows that $u_s$ fulfills the said properties in the claim for $0\leqslant s\leqslant\frac12$. Thus, we are left to check that for $\frac12<s\leqslant1$.

 \noindent$1^\circ$. By a direct computation and the volume-preserving property of the flow \eqref{eeforu}, we conclude that
  $$\dashint_{S^n}u_s^{2^\#}~d\mu_{S^n}=\frac{(2-2s)\zeta^{2^\#}
  \dashint_{S^n}u(T,u_0)^{2^\#}~d\mu_{S^n}+2s-1}{(2-2s)\zeta^{2^\#}+2s-1}=1.$$

  \noindent$2^\circ$. It follows from a direct computation, Lemma \ref{pos-property} and assumption (i) in Theorem  that
  $$\dashint_{S^n}fu_s^{2^\#}~d\mu_{S^n}=\frac{(2-2s)\zeta^{2^\#}
  \dashint_{S^n}fu(T,u_0)^{2^\#}~d\mu_{S^n}
  +(2s-1)\dashint_{S^n}f~d\mu_{S^n}}{(2-2s)\zeta^{2^\#}+2s-1}>0.$$

  \noindent$3^\circ$. We set
  $$w_s=\bigg[(2-2s)\big(\zeta u(T,u_0)\big)^{2^\#}+(2s-1)\bigg]^\frac{1}{2^\#}.$$
   Since the energy functional $E_f[u]$ is scale-invariant, we have
  \begin{equation}
    \label{scaleinvariant}
    E_f[u_s]=E_f[w_s]=\frac{\frac{1}{\omega_n}\int_{B^{n+1}}a_n|\nabla w_s|_{g_e}^2~dV_{g_e}+\dashint_{S^n}w_s^2~d\mu_{S^n}}
    {\bigg(\dashint_{S^n}fw_s^{2^\#}~d\mu_{S^n}\bigg)^\frac{n-1}{n}}
    :=\frac{I}{(II)^\frac{n-1}{n}}.
  \end{equation}
 {\bf Estimate of I}: A simple calculation shows that
 \begin{eqnarray*}
    |\nabla w_s|_{g_e}^2&=&\zeta^2(2-2s)^\frac{n-1}{n}\bigg[\frac{(2-2s)\big(\zeta u(T,u_0)\big)^{2^\#}}{(2-2s)\big(\zeta u(T,u_0)\big)^{2^\#}+(2s-1)}\bigg]^\frac{n+1}{n}|\nabla u(T,u_0)|_{g_e}^2\\
  &\leqslant&\zeta^2(2-2s)^\frac{n-1}{n}|\nabla u(T,u_0)|_{g_e}^2,
  \end{eqnarray*}
and by an elementary inequality, we get that
  $$w_s^2=\bigg[(2-2s)\big(\zeta u(T,u_0)\big)^{2^\#}+(2s-1)\bigg]^\frac{n-1}{n}
  \leqslant\zeta^2(2-2s)^\frac{n-1}{n}u^2(T,u_0)+(2s-1)^\frac{n-1}{n}.$$
  Combining the two estimates above yields
  \begin{equation}
    \label{estimateofI}
    I\leqslant\zeta^2(2-2s)^\frac{n-1}{n}E[u(T,u_0)]+(2s-1)^\frac{n-1}{n}.
  \end{equation}
  {\bf Estimate of $II$}: Notice that
  \begin{equation}
    \label{estimateofII}
    II=\dashint_{S^n}fw_s^{2^\#}~d\mu_{S^n}=\zeta^{2^\#}(2-2s)
    \dashint_{S^n}fu^{2^\#}(T,u_0)~d\mu_{S^n}+(2s-1)\dashint_{S^n}f~d\mu_{S^n}.
  \end{equation}
  Plugging \eqref{estimateofI} and \eqref{estimateofII} into \eqref{scaleinvariant} gives
  \begin{equation}\label{E_f[u_s]bound}
  E_f[u_s]=E_f[w_s]\leqslant\frac{\zeta^2(2-2s)^\frac{n-1}{n}E[u(T,u_0)]
  +(2s-1)^\frac{n-1}{n}}{\bigg(\zeta^{2^\#}(2-2s)
  \dashint_{S^n}fu^{2^\#}(T,u_0)~d\mu_{S^n}
  +(2s-1)\dashint_{S^n}f~d\mu_{S^n}\bigg)^\frac{n-1}{n}}.
  \end{equation}
  From the volume-preserving property of flow \eqref{eeforu} and \eqref{poslowerbound}, it follows that
  $$\alpha_1:=\bigg(\frac{1}{E_f[u_0]}\bigg)^\frac{n}{n-1}
  \leqslant\dashint_{S^n}fu^{2^\#}(T,u_0)~d\mu_{S^n}
  \leqslant\max_{S^n}|f|:=\alpha_2.$$
  Moreover, by the fact that $E_f[u(T,u_0)]\leqslant\beta_0$, we get
  $$E[u(T,u_0)]\leqslant\beta_0\bigg(\dashint_{S^n}fu^{2^\#}(T,u_0)~d\mu_{S^n}
  \bigg)^\frac{n-1}{n}\leqslant\beta_0\alpha_2^\frac{n-1}{n}.$$
  Substituting all the estimates above into \eqref{E_f[u_s]bound} yields
  $$E_f[u_s]\leqslant\frac{\zeta^2(2-2s)^\frac{n-1}{n}\beta_0\alpha_2^\frac{n-1}{n}
  +(2s-1)^\frac{n-1}{n}}
  {\bigg(\zeta^{2^\#}(2-2s)\alpha_1+(2s-1)\dashint_{S^n}f~d\mu_{S^n}\bigg)^\frac{n-1}{n}}.$$
  By letting $T\rightarrow+\infty$ in the estimate above, observing the fact \eqref{zetagoesto0} and the choice of $\beta_0$, we have
  $$\lim_{T\rightarrow+\infty}E_f[u_s]\leqslant\bigg(\dashint_{S^n}f~d\mu_{S^n}\bigg)^\frac{1-n}{n}<\beta_0.$$
  By choosing $T>0$ even larger , we thus complete the proof of claim.

   Therefore, we can conclude by the claim that $u_s\in L_{\beta_0}$ for all $0\leqslant s\leqslant1$. Moreover, by the definition of $u_s$, it is easy to see that $u_s=u_0\in L_{\beta_0}$ for $s=0$ and $u_s\equiv1$ for $s=1$. Hence, $u_s$ induces a contraction within $L_{\beta_0}$ and we complete the proof of (i).

   Recall that $f(p_i)>0$ at each critical point $p_i$, we have that $f(x)>0$ in a small neighborhood of $p_i$. Since the proof of (ii), (iii) and (iv) only requires the local information around each critical point $p_i$, it would have no difference from the case that $f$ is strictly positive. Hence, one can follow the exact same proof as that in \cite[Proposition 6.1]{xz}, we omit the details.
\end{proof}

\noindent{\bf\itshape Proof of Theorem \ref{main} and Corollary \ref{cor}}: Suppose the contrary, namely, $f$ cannot be realized as the boundary mean curvature of any conformal metric $g$ on the unit ball. A suitable choice of $\beta_0$ in part (i) of Proposition \ref{homotopy} shows that $L_{\beta_0}$ is contractible. In addition, the flow \eqref{eeforu} defines a homotopy equivalence of the set $\mathscr{E}_0=L_{\beta_0}$ with a set $\mathscr{E}_\infty$ whose homotopy type is that of a point with $n-\mbox{ind}_f(x)$ dimensional cells attached for every critical point $x$ of $f$ on $S^n$ where $f(x)>0$ and $\Delta_{S^n}f(x)<0$. It then follows from \cite[Theorem 4.3]{ch} that
\begin{equation}
  \label{systemequation}
  \sum_{i=0}^ns^im_i=1+(1+s)\sum_{i=0}^ns^ik_i
\end{equation}
holds for the Morse polynomials of $\mathscr{E}_0$ and $\mathscr{E}_\infty$, where $k_i\geqslant0$ and $m_i$ are given in \eqref{Morseindex1}. By equating the coefficients in the polynomials on the left and right hand side of \eqref{systemequation}, we obtain a set of non-trivial solutions of \eqref{Morseindex}, which violates the hypothesis in Theorem. We thus obtain the desired contradiction and the proof of Theorem \ref{main} is completed. Furthermore, by setting $s=-1$ in \eqref{systemequation} we can obtain \eqref{indexcounting} and thus the assertion in Corollary \ref{cor} holds.\qed
\subsection{Proof of Theorem \ref{main1}}
In view of \cite[Lemma 2.1]{ho}, we see that $u$ is a $G$-invariant function if the initial data $u_0\in X_f$ is a $G$-invariant function. Fix any $G$-invariant initial data $u_0\in X_f$, by the uniqueness of the solution of the flow \eqref{eeforu} and the decay of $E_f[u]$, we can assume that
\begin{equation}\label{energybound}
E_f[u(t)]<E_f[u_0],\quad\mbox{for}~~\forall t\in(0,+\infty),
\end{equation}
 Since we have assumed that case (ii) in Theorem \ref{cc} occurs to the corresponding sequential metrics $(g(t_k))_k$, the blow-up behavior of $(g(t_k))_k$ in Proposition \ref{sequentialblowup} will happen. In particular, it follows from Proposition \ref{sequentialblowup} that
\begin{equation}\label{integralonB_r(Q)}
\lim_{k\rightarrow+\infty}\omega_n^{-1}\int_{\partial^\prime B_r(Q)}fu_k^{2^\#}~d\mu_{S^n}=f(Q)=\frac{1}{\lambda_\infty},
\end{equation}
where $r>0$ is arbitrary and $Q$, depending on the choice of $u_0$, is the unique concentration point in Proposition \ref{sequentialblowup}, which also implies that for any $y\in S^n$, if $Q\notin B_r(y)$ for some $r>0$, then
\begin{equation}
  \label{integralonoutsideofB_r(Q)}
  \lim_{k\rightarrow+\infty}\int_{\partial^\prime B_r(y)}fu_k^{2^\#}~d\mu_{S^n}=0.
\end{equation}
Now, we split our argument into two cases

\noindent{\bf Case 1}. $\Sigma=\emptyset$. If this case happens, then we can find $\theta\in G$ such that $\theta(Q)\neq Q$. Since $f$ and $u_k$ are $G$-invariant, we conclude, by \eqref{integralonB_r(Q)} and change of variables, that
\begin{eqnarray*}
  \lim_{k\rightarrow+\infty}\omega_n^{-1}\int_{\partial^\prime B_r(\theta(Q))}fu_k^{2^\#}~d\mu_{S^n}
  &=&\lim_{k\rightarrow+\infty}\omega_n^{-1}\int_{\partial^\prime B_r(Q)}(f\circ\theta(y))(u_k\circ\theta(y))^{2^\#}~d\mu_{S^n}\\
  &=&\lim_{k\rightarrow+\infty}\omega_n^{-1}\int_{\partial^\prime B_r(Q)}fu_k^{2^\#}~d\mu_{S^n}\\
  &=&\frac{1}{\lambda_\infty}\neq0
\end{eqnarray*}
On the other hand, as $\theta(Q)\neq Q$, we can find $r>0$ small enough such that $Q\notin \partial^\prime B_r(\theta(Q))$. Then, by \eqref{integralonoutsideofB_r(Q)}, we have
$$\lim_{k\rightarrow+\infty}\omega_n^{-1}\int_{\partial^\prime B_r(\theta(Q))}fu_k^{2^\#}~d\mu_{S^n}=0,$$
which is a contradiction.

\noindent{\bf Case 2}. $\Sigma\neq\emptyset$ and $\max_\Sigma f\leqslant\dashint_{S^n}f~d\mu_{S^n}$: By Remark \ref{rk2}, we can choose the initial data $u_0\equiv1$ which is obviously a $G$-invariant function. If $Q\notin \Sigma$, then we can obtain a contradiction by repeating the argument in Case 1. Hence, we must have $Q\in \Sigma$. To proceed, we need a refined estimate upon the number $\lambda_\infty$.
By the decay of $E_f[u]$, we have
$$\frac{E[u_k]}{(\dashint_{S^n}fu_k^{2^\#}~d\mu_{S^n})^\frac{n-1}{n}}\leqslant E_f[u_1],$$
for all $k\geqslant1$, which implies, by sharp Sobolev trace inequality and volume-preserving property, that
\begin{equation*}
\dashint_{S^n}fu_k^{2^\#}~d\mu_{S^n}\geqslant \bigg(\frac{1}{E_f[u_1]}\bigg)^\frac{n}{n-1}.
\end{equation*}
From this, it follows that
\begin{eqnarray*}
  \lambda_k&=&E_f[u_k]\bigg(\dashint_{S^n}fu_k^{2^*}~d\mu_{S^n}\bigg)^{-\frac 1n}\\
  &\leqslant& E_f[u_1]\bigg(\frac{1}{E_f[u_1]}\bigg)^{-\frac{1}{n-1}}\\
  &=&\big(E_f[u_1]\big)^\frac{n}{n-1}.
\end{eqnarray*}
By letting $k\rightarrow+\infty$ in the inequality above and \eqref{energybound}, we have
$$\lambda_\infty\leqslant\big(E_f[u_1]\big)^\frac{n}{n-1}
<\big(E_f[u_0]\big)^\frac{n}{n-1}=\frac{1}{\dashint_{S^n}f~d\mu_{S^n}},$$
where we have used the fact $u_0\equiv1$ in the last equality. On the other hand, we know that $\lambda_\infty f(Q)=1$. Hence, we conclude that
$$f(Q)>\dashint_{S^n}f~d\mu_{S^n},$$
which contradicts with our assumption in Theorem \ref{main1}.

\subsection{Proof of Theorem \ref{main2}} Notice tht $\Sigma\neq\emptyset$. If $\max_\Sigma f\leqslant\dashint_{S^n}f~d\mu_{S^n}$, then we can repeat the argument in the case 2 of Theorem \ref{main1} to obtain the assertion. While $\max_\Sigma f>\dashint_{S^n}f~d\mu_{S^n}$, from \cite[Lemma 3.1]{ho}, we can choose an initial data $u_0$ which is invariant under (Sym1) or (Sym2) such that 1) $\dashint_{S^n}fu_0^{2^\#}~d\mu_{S^n}>0$; 2) for sufficiently small $\epsilon>0$ there holds
\begin{equation}\label{chioceofu_0}
E_f[u_0]<\frac{1}{\big(\max_\Sigma f\big)^\frac{n-1}{n}}+\epsilon,
\end{equation}
which implies, by the choice of $\beta$, that
$$E_f[u_0]<\frac{1}{\bigg(\dashint_{S^n} f~d\mu_{S^n}\bigg)^\frac{n-1}{n}}+\epsilon\leqslant\beta.$$
This shows that $u_0\in X_f$. Once we have this fact, the conclusions in Proposition \ref{asymptoticbehavior} will hold. In particular, at the corresponding concentration point $Q$, there holds $\Delta_{S^n}f(Q)\leqslant0$. In view of \cite[Proof of Theorem 1.2]{ho}, we can obtain that there exists a constant $\alpha>0$ such that
$$f(y)=\max_\Sigma f>f(Q)+\alpha.$$
Substituting this inequality into \eqref{chioceofu_0} yields
$$E_f[u_0]<\frac{1}{\big(f(Q)+\alpha\big)^\frac{n-1}{n}}+\epsilon
=\frac{1}{\big(f(Q)\big)^\frac{n-1}{n}}
\bigg(\frac{f(Q)}{(f(Q)+\alpha)}\bigg)^\frac{n-1}{n}+\epsilon.$$
Since $\epsilon$ is sufficiently small, we have
$$\epsilon\leqslant\frac{1}{\big(f(Q)\big)^\frac{n-1}{n}}
\bigg[1-\bigg(\frac{f(Q)}{(f(Q)+\alpha)}\bigg)^\frac{n-1}{n}\bigg],$$
which implies, by Proposition \ref{asymptoticbehavior}, that
$$E_f[u_0]<\big(f(Q)\big)^\frac{1-n}{n}=\lim_{k\rightarrow+\infty}E_f[u_k].$$
But this contradicts with the decay property of the energy functional $E_f[u]$. \qed

\section*{Acknowledgement}
  \noindent This project is supported by the ``Fundamental Research Funds for the Central Universities"

\end{document}